\documentclass{amsart}

\usepackage{amssymb}
\usepackage{epsfig}
\usepackage{amscd}
\usepackage{amsthm}
\usepackage[noadjust]{cite} 

\newtheorem{thm}{Theorem}[section]
\newtheorem{prop}[thm]{Proposition}
\newtheorem{lem}[thm]{Lemma}

\theoremstyle{definition}
\newtheorem{definition}[thm]{Definition}

\theoremstyle{remark}
\newtheorem{remark}[thm]{Remark}

\numberwithin{equation}{section}                                                                                                                                                                                                                                                                                                                                                                                                                                                                                                                                                                                                                                                                                                                                                                                                                                                                                                                                                                                                                                                                                                                                                                                                                                                                                                                                                                                                                                                                                                                                                                                                                                                                                                                                                                                                                                                                                                                                                                                                                                                                                                                                                                                                                                                                                                                                                                                                                                                                                                                                                                                                                                                                                                                                                                                                                                                                                                                                                                                                                                                                                                                                                                                                                                                                                                                                                                                                                                                                                                                                                                                                                                                                                                                                                                                                                                                                                                                                                                                                                                                                                                                                                                                                                                                                                                                                                                                                                                                                                                                                                                                                                                                                                                                                                                                                                                                                                                                                                                                                                                                                                                                                                                                                                                                                                                                                                                                                                                                                                                                                                                                                                                                                                                                                                                                                                                                                                                                                                                                                                                                                                                                                                                                                                                                                                                                                                                                                                                                                                                                                                                                                                                                                                                                                                                                                                                                                                                                                                                                                                                                                                                                                                                                                                                                                                                                                                                                                                                                                                                                                                                                                                                                                                                                                                                                                                                                                                                                                                                                                                                                                                                                                                                                                                                                                                                                                                                                                                                                                                                                                                                                                                                                                                                                                                                                                                                                                                                                                                                                                                                                                                                                                                                                                                                                                                                                                                                                                                                                                                                                                                                                                                                                                                                                                                                                                                                                                                                                                                                                                                                                                                                                                                                                                                                                                                                                                                                                                                                                                                                                                                                                                                                                                                                                                                                                                                                                                                                                                                                                                                                                                                                                                                                                                                                                                                                                                                                                                                                                                                                                                                                                                                                                                                                                                                                                                                                                                                                                                                                                                                                                                                                                                                                                                                                                                                                                                                                                                                                                                                                                                                                                                                                                                                                                                                                                                                                                                                                                                                                                                                                                                                                                                                                                                                                                                                                                                                                                                                                                                                                                                                                                                                                                                                                                                                                                                                                                                                                                                                                                                                                                                                                                                                                                                                                                                                                                                                                                                                                                                                                                                                                                                                                                                                                                                                                                                                                                                                                                                                                                                                                                                                                                                                                                                                                                                                                                                                                                                                                                                                                                                                                                                                                                                                                                                                                                                                                                                                                                                                                                                                                                                                                                                                                                                                                                                                                                                                                                                                                                                                                                                                                                                                                                                                                                                                                                                                                                                                                                                                                                                                                                                                                                                                                                                                                                                                                                                                                                                                                                                                                                                                                                                                                                                                                                                                                                                                                                                                                                                                                                                                                                                                                                                                                                                                                                                                                                                                                                                                                                                                                                                                                                                                                                                                                                                                                                                                                                                                                                                                                                                                                                                                                                                                                                                                                                                                                                                                                                                                                                                                                                                                                                                                                                                                                                                                                                                                                                                                                                                                                                                                                                                                                                                                                                                                                                                                                                                                                                                                                                                                                                                                                                                                                                                                                                                                                                                                                                                                                                                                                                                                                                                                                                                                                                                                                                                                                                                                                                                                                                                                                                                                                                                                                                                                                                                                                                                                                                                                                                                                                                                                                                                                                                                                                                                                                                                                                                                                                                                                                                                                                                                                                                                                                                                                                                                                                                                                                                                                                                                                                                                                                                                                                                                                                                                                                                                                                                                                                                                                                                                                                                                                                                                                                                                                                                                                                                                                                                                                                                                                                                                                                                                                                                                                                                                                                                                                                                                                                                                                                                                                                                                                                                                                                                                                                                                                                                                                                                                                                                                                                                                                                                                                                                                                                                                                                                                                                                                                                                                                                                                                                                                                                                                                                                                                                                                                                                                                                                                                                                                                                                                                                                                                                                                                                                                                                                                                                                                                                                                                                                                                                                                                                                                                                                                                                                                                                                                                                                                                                                                                                                                                                                                                                                                                                                                                                                                                                                                                                                   
\begin{document}
\title{The Euler characteristics of generalized Kummer schemes}
\author{Junliang Shen}

\address{Departement Mathematik, ETH Z\"urich}
\email{junliang.shen@math.ethz.ch}
 \begin{abstract} We compute the Euler characteristics of the generalized Kummer schemes associated to $A\times Y$, where $A$ is an abelian variety and $Y$ is a smooth quasi-projective variety. When $Y$ is a point, our results prove a formula conjectured by Gulbrandsen.
\end{abstract} 

\maketitle

\section{Introduction}
We work over the complex numbers $\mathbb{C}$. Let $X$ be the smooth variety $A\times Y$ where $Y$ is an $r$-dimensional smooth quasi-projective variety and $A$ is a $g$-dimensional abelian variety. There is a natural morphism from the Hilbert scheme of points on $X$ to the abelian variety
\begin{equation}
\pi_n: \mathrm{Hilb}^n(X) \rightarrow A
\end{equation}
which is the composition of the map $ \mathrm{Hilb}^n(X) \rightarrow \mathrm{Sym}^n(A)$ and the addition map $\mathrm{Sym}^n(A) \rightarrow A$.\footnote[1]{Since we mainly concern Euler characteristics, we can always work with the reduced scheme structure in this paper.} It is clear that $\pi_n$ gives an isotrivial fibration which shows that the Euler characteristic of $\mathrm{Hilb}^{n}(X)$ is 0. The fiber over the origin  $O_A \in A$ is defined to be
\[
K_n(X) := \pi_n^{-1}(O_A).
\]
We call the construction above the \textit{generalized Kummer construction} of $X=A \times Y$. When $A$ is an abelian surface and $Y$ is a point, this construction gives a group of interesting examples of holomorphic symplectic varieties, which are the generalized Kummer varieties introduced by Beauville \cite{B}. However, it should be mentioned that the scheme $K_n(X)$ is always singular when $\textup{dim}(X)>2$ and $n>3$.

The Euler characteristics of $K_n(X)$ are completely determined by the following formula.
\begin{thm}\label{mainthm}
With the notation as above, we have
\begin{equation*}
\mathrm{exp}\Big{(} \sum_{n \geq 1} \frac{\chi{(K_n(X))}}{n^{2g}} \cdot t^n  \Big{)} = \Big{(} \sum_{k\geq 0} P_{r+g}(k)\cdot t^k\Big{)}^{\chi(Y)}.
\end{equation*}
Here $g \geq 1$ and $P_m(k)$ is the number of $m$-dimensional partitions of $k$. 
\end{thm}

Note that partition counts have been previously related to the Euler characteristics of Hilbert schemes in \cite{cheah} and to Donaldson--Thomas invariants in \cite{MNOP}.

When $Y$ is a point, the theorem shows that
\begin{equation*}
\mathrm{exp}\Big{(} \sum_{n \geq 1} \frac{\chi{(K_n(A))}}{n^{2g}} \cdot t^n  \Big{)} =  
 \sum_{k\geq 0} P_{g}(k)\cdot t^k.
\end{equation*}
This formula was conjectured by Gulbrandsen \cite{Gul2}. In particular, the case of $g=2$ gives the formula for the Euler characteristics of generalized Kummer varieties
\begin{equation*}
\chi (K_n(A)) = n^3 \sum_{d|n} d,
\end{equation*}
which has been proven by G\"ottsche \cite{Lecture}, G\"ottsche and Soergel \cite{GS}, Debarre \cite{D} and Gulbrandsen \cite{Gul1} in different ways. 

When $g=1$ and $r=1$, the variety $K_n(X)$ is also smooth. G\"ottsche computed the Euler characteristic $\chi(K_n(X))$ when $Y = \mathbb{P}^1$ in \cite{Lecture} Chapter 2.4 by using the Weil conjecture. Our method extends G\"ottsche's results to any fibration $X \rightarrow E$ where $E$ is an elliptic curve and the fiber is a smooth curve $Y$ (see Remark \ref{map}),
\[
\chi (K_n(X)) = \chi(Y)\cdot n \sum_{d|n} d.
\]

When $\textup{dim}(X) = 3$, by the MacMahon's product formula for the generating series of 3-dimensional partitions \cite{Com}, we can get the following interesting formula, \footnote[2]{After my talk on the results of this paper at the workshop ``Motivic invariants related to K3 and abelian geometries" in Berlin, I was informed by Andrea Ricolfi that he obtained the formula when $Y$ is a point and $A$ is an abelian 3-fold independently.} 
\[
\chi(K_n(A\times Y)) = \chi(Y) \cdot n^{2g-1}\sum_{d|n}d^2. 
 \]
By \cite{Gul2}, this formula computes the Donaldson--Thomas invariants in degree zero for abelian 3-folds,
\[
\textup{DT}_{n,0}(A) = \frac{(-1)^{n-1}}{n^6}\chi(K_n(A)) = \frac{(-1)^{n-1}}{n}\sum_{d|n}d^2. \]

The full Donaldson--Thomas theory of curves on abelian 3-folds is discussed in \cite{BOPY}. Invariants in primitive classes are conjectured explicitly in terms of Jacobi forms.

 The motivic theory of the generalized Kummer schemes has recently been established in a joint work with Morrison \cite{MS}, which gives a motivic refinement of the main theorem.\\
 \\
Our proof of the theorem follows 2 steps:\\
{\bf Step 1.} Use cut-and-paste to show that the Euler characteristic of $K_n(A\times Y)$ does not depend on the choice of the $g$-dimensional abelian variety $A$.\\
{\bf Step 2.} Generalize the method in \cite{Gul1} to prove the theorem when $A = E \times B$ where $E$ is an elliptic curve and $B$ is a $(g-1)$-dimensional abelian variety.

\subsection*{Acknowledgement} I would like to thank my advisor Rahul Pandharipande for his support, encouragement and helpful conversations. Thanks also to Andrew Morrison, Georg Oberdieck, and Qizheng Yin for related discussions, to Jim Bryan for an inspiring talk in the moduli seminar at ETH Z\"urich, and to
the referee for comments and suggestions.

This work was carried out in the group of Pandharipande at ETH Z\"urich, supported by grant ERC-2012-AdG-320368-MCSK .

\section{The geometry of the generalized Kummer construction}
Let $\rho: \mathrm{Hilb}^n (X) \rightarrow \mathrm{Sym}^n(A) $ be the composition of the Hilbert--Chow morphism $\mathrm{Hilb}^n(X) \rightarrow \mathrm{Sym}^n(X)$ and the projection $\mathrm{Sym}^n(X)\rightarrow \mathrm{Sym}^n(A)$, and $f: \mathrm{Sym}^n(A) \rightarrow A$ be the addition map.
Thus we have $\pi_n = f \circ \rho$. There is a standard stratification
\begin{equation*}
\mathrm{Sym}^n(A) = \coprod_{\alpha} A^{(n)}_{\alpha}.
\end{equation*}
Here $\alpha$ runs through all (2-dimensional) partitions of $n$, and if we write $\alpha$ to be the partition $n=n_1+n_2+\cdots+n_l$, the corresponding stratum is
\begin{equation*}
A^{(n)}_{\alpha}=\Big{\{} \sum_{i=1}^{l}n_i[a_i] \in \mathrm{Sym}^n(A) \,    \Big{|} \, a_i\in A, a_i \neq a_j \textup{ for } i\neq j   \Big{\}}.
\end{equation*}

We study the morphism $\rho$ over each $A^{(n)}_{\alpha}$. First we introduce the following definition.
\begin{definition}
For a smooth quasi-projective variety $M$, we denote by $\mathrm{Hilb}^{n}(M\times \mathbb{C}^m)_{\{0\}}$ the subscheme of $\mathrm{Hilb}^{n}(M\times \mathbb{C}^m)$ consisting of subschemes supported on $M\times \{ 0 \} $, \textit{i.e.}
\[
\mathrm{Hilb}^{n}(M\times \mathbb{C}^m)_{\{0\}} = \Big{\{} \xi \in \mathrm{Hilb}^{n}(M\times \mathbb{C}^m) \,\Big{|}\, \mathrm{Supp}(\xi) \subset M \times \{0\} \Big{\}}.
\]
\end{definition}
\begin{remark}
\begin{enumerate}
\item[1.] When $M$ is a point, the above definition exactly gives the punctual Hilbert scheme. 
\item[2.] For an $m$-dimensional smooth projective variety $M'$ and a point $p \in M'$, we can also define similarly the scheme 
\[\mathrm{Hilb}^{n}(M\times M')_{\{p\}}:= \Big{\{} \xi \in \mathrm{Hilb}^{n}(M\times M') \, \Big{|} \, \mathrm{Supp}(\xi) \subset M \times \{p\} \Big{\}}.
\]
It is easy to see that it does not depend on the choice of $M'$ and the point $p$. We have  $\mathrm{Hilb}^{n}(M\times M')_{\{p\}} \cong \mathrm{Hilb}^{n}(M\times \mathbb{C}^m)_{\{0\}}$.
\end{enumerate}
\end{remark}

Now we fix a partition $\alpha = (n_1, n_2, \dots,n_l)$ of $n$. The morphism \[\rho|_{A^{(n)}_\alpha}: \rho^{-1}(A^{(n)}_{\alpha}) \rightarrow A^{(n)}_{\alpha}\] is a fibration with fiber 
\begin{equation}\label{fiber}
F_{\alpha} \cong \mathrm{Hilb}^{n_1}(Y\times \mathbb{C}^g)_{\{0\}} \times \cdots \times \mathrm{Hilb}^{n_l}(Y\times \mathbb{C}^g)_{\{0\}}.
\end{equation}  
Hence we obtain the following lemma.
\begin{lem}\label{euler}
The Euler characteristic of $K_n(X)$ can be expressed as follows
\begin{equation}\label{euler2}
\chi(K_n(X)) = \sum_{\alpha} \chi(A^{(n)}_{\alpha , 0}) \cdot \chi(F_\alpha),
\end{equation}
where $\alpha$ runs through all partitions of $n$ and $A^{(n)}_{\alpha,0}=f^{-1}(O_A)\bigcap A^{(n)}_{\alpha}$.
\end{lem}

\section{Step 1}
\begin{prop}\label{prop1}
Let $A$ and $A'$ be two $g$-dimensional abelian varieties and $Y$ be an $r$-dimensional smooth quasi-projective variety , we have
\[
\chi(K_n(A\times Y)) = \chi( K_n(A'\times Y)).
\]
\end{prop}

\begin{proof}
Since the underlying topological spaces of $A$ and $A'$ are $2g$-dimensional tori. We can construct a homeomorphism
\[
\phi: A \rightarrow A'
\] which preserves the group law, \textit{i.e}.
\[
\phi(a+_A b) = \phi(a)+_{A'} \phi(b). 
\]
The map $\phi$ induces a homeomorphism between $\mathrm{Sym}^{n}(A)$ and $\mathrm{Sym}^{n}(A')$, and homeomorphisms between the corresponding strata $A^{(n)}_{\alpha}$ and $A'^{(n)}_{\alpha}$ as well. Moreover, since $\phi$ preserves the group law, it also induces homeomorphisms between $A^{(n)}_{\alpha,0}$ and $A'^{(n)}_{\alpha, 0}$. Therefore we have
\begin{equation}\label{strata}
\chi(A^{(n)}_{\alpha,0}) = \chi(A'^{(n)}_{\alpha, 0}).
\end{equation}
The proposition is a consequence of Lemma \ref{euler}, (\ref{strata}) and the fact that $F_\alpha$ does not depend on the choice of the abelian variety (see (\ref{fiber})).
\end{proof}

\begin{remark}\label{map}
If $X \rightarrow A$ is a fibration with smooth fiber $Y$, we can also define the generalized Kummer construction associated to $X \rightarrow A$ by the same way:
\[
\textup{Hilb}^n(X) \rightarrow \textup{Sym}^n(X) \rightarrow  \textup{Sym}^n(A) \rightarrow A. \]
As a consequence of Lemma \ref{euler}, we have
\[
\chi(K_n(X)) = \chi(K_n(A \times Y))
\]
since both Euler characteristics are equal to the right hand side of (\ref{euler2}).
\end{remark}

\section{Combinatoric relations}
In this section, we generalize the combinatorics used in \cite{Gul1}. A partition $\alpha$ of $n$ may be written as $\alpha=(1^{\alpha_1}2^{\alpha_2}\cdots n^{\alpha_n})$ indicating the number of times each positive integer occurs in $\alpha$. 
\begin{definition}
We define a real number  $e(\alpha)$ for each partition $\alpha$ of $n$ by the following recursion:
\begin{enumerate}
\item[(1).] $e(n^1) = n^2$.
\item[(2).] If $\alpha =(1^{\alpha_1}2^{\alpha_2}\cdots n^{\alpha_n}) \neq (n^1)$, then
\[
e(\alpha) = - \sum_{i} \frac{n}{n-i}e(1^{\alpha_1}\dots i^{\alpha_i-1}\dots).
\]
\end{enumerate}
\end{definition}

\begin{prop}\label{prop2}
Let $\{a_n\}_{n\geq 1}$ and $\{b_n\}_{n\geq 1}$ be two sequences of variables. For convenience we also assume $b_0 =1$. The following two conditions are equivalent.
\begin{enumerate}
\item[(1).]The variables $\{a_n\}$ and $\{b_n\}$ satisfy the relations:
\[
a_n = \sum_{\alpha}e(\alpha) \cdot b_1^{\alpha_1}b_2^{\alpha_2}\cdots b_n^{\alpha_n},
\]
where $\alpha = (1^{\alpha_1}2^{\alpha_2}\cdots n^{\alpha_n})$ runs through all partitions of $n$.
\item[(2).]The variables $\{a_n\}$ and $\{b_n\}$ satisfy the equation:
\begin{equation}\label{relation}
\mathrm{exp}\Big{(}  \sum_{n \geq 1} \frac{a_n}{n^2}\cdot t^n  \Big{)} = \sum_{k\geq 0} b_kt^k.
\end{equation}
\end{enumerate}
\end{prop}
\begin{proof}
Since the variables $a_n$ are uniquely determined by $b_n$ in both condition (1) and (2), it suffices to prove the following statement.

($\bullet$) Assume $\{a_n\}$ satisfy (\ref{relation}), then they also satisfy the relations in (1).

By applying the operator $t\frac{d}{dt}$ on both sides of (\ref{relation}), we have
\[
\Big{(}  \sum_{n\geq 1} \frac{a_n}{n}\cdot t^n  \Big{)}\Big{(}  \sum_{m\geq 0}b_mt^m  \Big{)} =  \sum_{k \geq 0} kb_kt^k.
\]
Hence we get 
\begin{equation}\label{induction}
a_n = n\Big{(}-\frac{a_{n-1}}{n-1}b_1 -\frac{a_{n-2}}{n-2}b_2 -\cdots -a_1b_{n-1} + nb_n\Big{)}
\end{equation}
by comparing the coefficient of $t^n$. The statement ($\bullet$) is obtained by induction (using (\ref{induction})) and the recursion of $e(\alpha)$.
\end{proof}

\begin{remark}
Assume $a_n = n \sum_{d|n}d$ and $b_n=P_2(n)$. By the well-known formula
\[
\sum_{n\geq 0} P_2(k)t^k = \prod_{k\geq1}\Big{(}  \frac{1}{1-t^k}        \Big{)},
\] 
it is clear they satisfy the equation (\ref{relation}). Thus from Proposition \ref{prop2} we know that
\[
n \sum_{d|n}d = \sum_{\alpha}\prod_{i}P_2(i)^{\alpha_i}e(\alpha).
\]This is the combinatoric formula shown in \cite{Gul1} Section 2.3.

Similarly, by MacMahon's formula 
\[
\sum_{n\geq 0} P_3(k)t^k = \prod_{k\geq1}\Big{(}  \frac{1}{1-t^k} \Big{)}^k,
\] 
we know that $a_n = n\sum_{d|n}d^2$ and $b_n= P_3(n)$ satisfy (\ref{relation}). Hence we also have
\[
n \sum_{d|n}d^2 = \sum_{\alpha}\prod_{i}P_3(i)^{\alpha_i}e(\alpha).
\]
\end{remark}

\section{Step 2}
By Proposition \ref{prop1}, we may assume the $g$-dimensional abelian variety $A$ splits as $E\times B$ where $E$ is an elliptic curve and $B$ is a $(g-1)$-dimensional abelian variety. In this section we follow the idea in \cite{Gul1}.

Similar as in Section 1, we have the standard stratification of $\mathrm{Sym}^n(E)$,
\[
\mathrm{Sym}^n(E) = \coprod_{\alpha} E^{(n)}_{\alpha},
\]
and the addition map $g: \mathrm{Sym}^n(E) \rightarrow E$.
We denote by $P$ the preimage of $O_E$, \textit{i.e.} $P = g^{-1}(O_E)$. Moreover, we define $E^{(n)}_{\alpha, 0} = P \bigcap E^{(n)}_\alpha$.

\begin{lem}[{\cite{Gul1}} Section 5]\label{e}
We have \[
\chi(E^{(n)}_{\alpha, 0}) = e(\alpha).
\]
\end{lem}
This lemma is proven by showing that $\chi(E^{(n)}_{\alpha , 0})$ satisfies the same recursion formula as $e(\alpha)$.

Now we consider the projection $h: \mathrm{Hilb}^{n}(E\times \mathbb{C}^m) \rightarrow \mathrm{Sym}^{n}(E)$ and define the subscheme
\[
W^n_m: = h^{-1}(P) \cap \mathrm{Hilb}^{n}(E\times \mathbb{C}^m)_{\{0\}}.
\]
\begin{lem}\label{w}
The Euler characteristics $\chi(W^n_m)$ satisfy the following equation
\[
\mathrm{exp}\Big{(} \sum_{n \geq 1} \frac{\chi{(W^n_m)}}{n^2} \cdot t^n  \Big{)} =  
 \sum_{k\geq 0} P_{m+1}(k)\cdot t^k.
 \]
\end{lem}
\begin{proof}
By the same argument as in \cite{Gul1} Lemma 4.3, we get 
\[
\chi(W^n_m)=\sum_{\alpha}\prod_{i}P_{m+1}(i)^{\alpha_i}\chi(E^{(n)}_{\alpha,0}).
\]
The lemma is then a consequence of Proposition \ref{prop2} and Lemma \ref{e}.
\end{proof}

We relate $\chi(W^n_m)$ to $\chi(K_n(X))$. Since $X = E\times B \times Y$, there is a projection
\[
p: K_n(X) \rightarrow \mathrm{Sym}^n(B\times Y).
\]
\begin{lem}\label{final}
\begin{enumerate}
\item[(1).]If a point $Q\in \mathrm{Sym}^n(B\times Y)$ is not of the form $n\cdot [b]$ for some point $b\in B\times Y$, then\[
\chi(p^{-1}(Q)) =0.
\] 
\item[(2).]The Euler characteristic of $K_n(X)$ is
\[
\chi(K_n(X)) = n^{2g-2}\chi(Y)\cdot \chi(W^n_{g+r-1}).
\]
\end{enumerate}
\end{lem}
\begin{proof}
The proof of \cite{Gul1} Lemma 4.2 works for (1). Now we prove (2). We know from (1) that only 
\[
p^{-1}\Big{\{}n\cdot [t\times y] \in \mathrm{Sym}^{n}(B\times Y)     \,    \Big{|}  \, t \textup{ is an $n$-torsion point on $B$}, \,y\in Y          \Big{\}}
\] contributes to the Euler characteristic $\chi(K_n(X))$. Since the locus
\[\Big{\{}n\cdot [t\times y] \in \mathrm{Sym}^{n}(B\times Y)    \,     \Big{|}  \, t \textup{ is an $n$-torsion point on } B,\, y\in Y          \Big{\}}    \subset \mathrm{Sym}^{n}(B\times Y)     \]
is isomorphic to $n^{2g-2}$ copies of $Y$, and the fiber of $p$ over each point $n\cdot [t\times y] \in \mathrm{Sym}^{n}(B\times Y)$ is isomorphic to $W^n_{g+r-1}$, we get the formula
\[
\chi(K_n(X)) = n^{2g-2}\chi(Y)\cdot \chi(W^n_{g+r-1}).\qedhere
\]
\end{proof}

The theorem is deduced by Lemma \ref{w} and Lemma \ref{final} (2).

\begin{remark}
One can also consider the natural $A$-action on the total Hilbert scheme $\mathrm{Hilb^n(X)}$ induced by the $A$-translate on $X$. It is easy to see that the quotient stack $K^{\textup{Quot}}_n(X)$ is exactly the global quotient of $K_n(X)$ by the finite group $A[n]$( $n$-torsion points on $A$), \textit{i.e.}
\[
K_n^\textup{Quot}(X) \cong \big{[} K_n(X)/ A[n] \big{]}.
\]
Hence the invariants $\chi(K_n(X))/n^{2g}$ in the formula of the main theorem can be viewed as Euler characteristics of the stacks $K_n^{\textup{Quot}}(X)$.
\end{remark}

\nocite{*} 
\bibliographystyle{plain}
\bibliography{Meinbib}
\end{document}